\documentclass[a4paper,12pt,reqno]{amsart}

\usepackage{amsmath}
\usepackage{amssymb}
\usepackage{amsfonts}
\usepackage{graphicx}
\usepackage[colorlinks=true, pdfstartview=FitV, linkcolor=blue,
citecolor=blue, urlcolor=blue]{hyperref}
\renewcommand\eqref[1]{(\ref{#1})} 

\usepackage{times}
\usepackage{microtype}
\usepackage{cite}

\usepackage[margin=1.2in]{geometry}

\graphicspath{ {images/} }

\newtheoremstyle{theorem}
{10pt}          
{10pt}  
{\sl}  
{\parindent}     
{\bf}  
{. }    
{ }    
{}     
\theoremstyle{theorem}

\numberwithin{equation}{section}
\theoremstyle{plain}
\newtheorem{thm}{Theorem}[section]

\newtheorem{lem}[thm]{Lemma}

\newtheorem{remark}[thm]{Remark}

\newcommand{\R}{\mathbb R}

\DeclareMathOperator{\dist}{dist}

\DeclareMathOperator*{\essinf}{ess\,inf}
\DeclareMathOperator*{\esssup}{ess\,sup}

\DeclareMathOperator{\Div}{div}
\newcommand{\grad}{\nabla}

\DeclareMathOperator{\op}{op}

\newcommand{\pp}{{p(\cdot)}}
\newcommand{\cpp}{{p'(\cdot)}}
\newcommand{\Lp}{L^{p(\cdot)}}
\newcommand{\Pp}{\mathcal P}

\newcommand{\qq}{{q(\cdot)}}
\newcommand{\rr}{{r(\cdot)}}
\newcommand{\sst}{{s(\cdot)}}

\newcommand{\Lq}{L^{q(\cdot)}}

\newcommand{\ps}{{p^*(\cdot)}}

\def\Xint#1{\mathchoice
{\XXint\displaystyle\textstyle{#1}}%
{\XXint\textstyle\scriptstyle{#1}}%
{\XXint\scriptstyle\scriptscriptstyle{#1}}%
{\XXint\scriptscriptstyle\scriptscriptstyle{#1}}%
\!\int}
\def\XXint#1#2#3{{\setbox0=\hbox{$#1{#2#3}{\int}$ }
\vcenter{\hbox{$#2#3$ }}\kern-.6\wd0}}

\def\dashint{\Xint-}
\def\avgint{\Xint-}

\title[Hardy-Leray inequalities in variable Lebesgue spaces]
{Hardy-Leray inequalities \\ in variable Lebesgue spaces}

\author[D. Cruz-Uribe]{David Cruz-Uribe, OFS}
\address{
	David Cruz-Uribe, OFS:
	\endgraf
	Department of Mathematics
	\endgraf
	University of Alabama
	\endgraf
	Tuscaloosa, AL 35487
	\endgraf
	USA
	\endgraf
	{\it E-mail address} {\rm dcruzuribe@ua.edu}
}

\author[D. Suragan]{Durvudkhan Suragan}
\address{
	Durvudkhan Suragan:
	\endgraf
	Department of Mathematics
	\endgraf
 Nazarbayev University
	\endgraf
	53 Kabanbay Batyr Ave, Astana 010000
	\endgraf
	Kazakhstan
	\endgraf
	{\it E-mail address} {\rm durvudkhan.suragan@nu.edu.kz}}

\subjclass[2010]{26D10, 35A23, 39B62, 42B35}
\keywords{Hardy-Leray inequality, Rellich inequality, Hardy-Sobolev inequality,
  Gagliardo-Nirenberg inequality, Stein-Weiss inequality, variable
  Lebesgue spaces}

\thanks{The authors were partially supported by a Simons
  Foundation Travel Support for Mathematicians Grant.  This research was funded by Nazarbayev University under Collaborative Research Program Grant 20122022CRP1601. No new data
  was collected or generated during the course of this research.  This
  project was begun when the authors met at the Department of
  Mathematics at the University of Alabama in the Fall of 2022.  The
  second author
  would like to thank the university for their support and
  hospitality. }

\date{March 23, 2023}
\begin{document}

\begin{abstract}
  In this paper, we prove the Hardy-Leray inequality and related
  inequalities in variable Lebesgue spaces.  Our proof is based on a
  version of the Stein-Weiss inequality in variable Lebesgue spaces
  derived from two weight inequalities due to Melchiori and Pradolini. We also discuss an application of our results to establish an existence result for the degenerate $\pp$-Laplace operator. 
\end{abstract}

\maketitle

\section{Introduction and main results}
In this paper we consider the problem of extending the classical
Hardy-Leray inequality and related inequalities to the scale of
variable Lebesgue spaces.  The Hardy-Leray inequality, sometimes
referred to simply as the Hardy inequality, states that        
for $1<p<\infty$, $p\neq n$, and compactly supported smooth functions $f$,
\begin{equation}\label{Hardy-Leray}
  \bigg\|\frac{f}{|\cdot|}\bigg\|_{L^{p}(\mathbb{R}^{n})}\
  \leq 
        C\|\nabla f\|_{L^{p}(\mathbb{R}^{n})}.
 \end{equation}
 Originally, this was proved by Hardy~\cite{zbMATH02591430} in one dimension
 when $p=2$, and extended to higher dimensions by
 Leray~\cite{MR1555394}.  This is a special case of Caffarelli, Kohn,
 and Nirenberg~\cite{MR768824}, who proved the following general
 result.  Given $p,\,q\geq 1$ and exponents $a,\,b$ such that
 \[ 0 \leq b-a \leq 1, \quad \frac{-n}{p}<b, \quad \frac{-n}{q} <
   a, \quad \text{and} \quad
   \frac{1}{p}-\frac{1}{q} = \frac{1}{n} + \frac{a-b}{n}, \]
 then
 \[ \||\cdot|^a f \|_{L^{q}(\mathbb{R}^{n})}
   \leq
   C\||\cdot|^b \grad f\|_{L^{p}(\mathbb{R}^{n})}. \]
 Note that when $a=b=0$ this reduces to the classical Hardy-Sobolev
 inequality~\cite[Theorem~7.10]{MR1814364}. 

 Similar results hold for fractional powers of the Laplacian.
Let $\mathcal{S}$ be  the Schwartz space  of rapidly
decaying smooth functions in $\mathbb{R}^{n}$. 
The fractional Laplacian of a function can be defined by the following nonlocal operator in $\mathbb{R}^{n}$ for all $f\in\mathcal{S}$,
    \begin{equation}\label{fraclap}
    (-\Delta)^{s}f(x):=\frac{C(n,s)}{2}\int_{\mathbb{R}^{n}}\frac{2f(x) - f(x + y) - f(x -y)}{|y|^{n+2s}}dy, \;  0<s<1,
\end{equation}
where $C(n,s) > 0$ is a so-called normalization constant.
It  is known that (see e.g. \cite{MR2944369}) if $$(-\Delta)^{s}:\mathcal{S}\rightarrow L^{2}(\mathbb{R}^{n})$$ is the
fractional Laplacian given by \eqref{fraclap}. Then for any $f\in\mathcal{S}$, 
\begin{equation*}
     (-\Delta)^{s}f(x)=\mathcal{F}^{-1}\left(|\xi|^{2s}\mathcal{F}{f}\right)(x),\;\;\;x\in\mathbb{R}^{n},
\end{equation*}
where $\mathcal{F}$ is the Fourier transform of a function.

 Fix
$0<s\leq 1$; then for $1\leq p<\infty$ we have the fractional
Hardy-Rellich inequality
\[ \||\cdot|^{2s} f\|_{L^{p}(\mathbb{R}^{n})}
  \leq
  C\|(-\Delta)^s f \|_{L^{p}(\mathbb{R}^{n})}. \]
This is a special case of the fractional Hardy-Sobolev inequality:
given $p,\,q\geq 1$ and exponents $a,\,b$ such that
 \[ 0 \leq b-a \leq 2s, \quad \frac{-n}{p}<b, \quad \frac{-n}{q} <
   a, \quad \text{and} \quad
   \frac{1}{p}-\frac{1}{q} = \frac{2s}{n} + \frac{a-b}{n}, \]
 then
 \[ \||\cdot|^a f \|_{L^{q}(\mathbb{R}^{n})}
   \leq
   C\||\cdot|^b (-\Delta)^s f\|_{L^{p}(\mathbb{R}^{n})}. \]
 For these and related results,
 see~\cite{MR2944369,MR436854,MR3626031,deNitti-Djitte,MR2469027}. 

 \medskip
 
The variable Lebesgue spaces $\Lp(\Omega)$  are a generalization of the classical
Lebesgue spaces gotten by replacing the constant exponent $p$ by an
exponent function $\pp : \Omega \rightarrow [1,\infty)$.  Intuitively,
they consist of all the measurable functions $f$ that
satisfy
\[ \int_\Omega |f(x)|^{p(x)}\,dx < \infty.  \]
These functions form a Banach function space when equipped with the Luxemburg
norm
\[ \|f\|_\pp = \|f\|_{\Lp(\Omega)} = \inf\bigg\{ \lambda > 0\,:\,
  \int_\Omega \bigg(\frac{|f(x)|}{\lambda}\bigg)^{p(x)}\,dx \leq 1
  \bigg\}. \]
These spaces have been studied for many years; one of the
original motivations was the study PDEs and the calculus of
variations over these
function spaces.  For more details,
see~\cite{CUF13,diening-harjulehto-hasto-ruzicka2010, MR3379920} and
the references they contain.

Because of this, there has been considerable work done in proving
Hardy-type inequalities in the variable Lebesgue spaces.  
A one-dimensional version of
Hardy's inequality was first obtained by Kokilashvili and Samko in
\cite{KokSam04-1D}.  The one-dimensional case was also considered in the papers
\cite{DieSam07-ext1D} and \cite{MCMO07}.  In higher dimensions, for a
bounded domain $\Omega$, 
in \cite{HHK05} the authors proved that 

$$
\left\|\frac{f}{\delta(\cdot)}\right\|_{\Lp(\Omega)} \leq C\left\|\nabla 
f\right\|_{\Lp(\Omega)}, 
$$
where $\delta(x)=\operatorname{dist}(x, \partial \Omega)$.   On $\R^n$
and on bounded domains, the
Hardy-Sobolev inequality
\[ \|f\|_\qq \leq \|\grad f\|_\pp, \]
where $\frac{1}{p(x)}-\frac{1}{q(x)} = \frac{1}{n}$ for all $x$, has been studied by
a number of authors:  see\cite[Section~6.4]{CUF13} for details and
references.  

In this paper, we prove the Hardy-Leray inequality and related results
on the variable Lebesgue spaces over $\R^n$.   We will state our main
results here, but for brevity
we will defer the definition of our  notation (which is standard)
to the next section.    Our first result is a generalization of
inequality~\eqref{Hardy-Leray}.

\begin{thm}[Hardy-Leray inequality]\label{thm1}
 Given $\pp \in \Pp(\R^n)$, if  $\pp \in LH(\R^n)$ and $1 < p_- \leq p_{+}<n$, then for all $f\in C_0^\infty(\R^n)$, 
    \begin{equation*}
        \left\|\frac{f}{|\cdot|}\right\|_{p(\cdot)}\leq C\|\nabla f\|_{p(\cdot)},
    \end{equation*}
    where $C$ is a positive constant independent of $f$. 
\end{thm}
	
\begin{remark}
  By using the theory of variable exponent Sobolev spaces, these
  results can be extended to functions $f$ whose (weak) gradient
  $\nabla f \in \Lp(\R^n)$.   Here and below we leave these extensions
  to the interested reader.
\end{remark}

Theorem~\ref{thm1} is a special case of the following
result.

\begin{thm}[Hardy-Sobolev inequality]\label{Hardy-Sob}
Fix exponents $\pp,\, \qq \in \Pp(\R^n)$ such that $\pp,\, \qq \in LH(\R^n)$,
$1<p_-\leq p_+<\infty$, $1<q_-\leq q_+<\infty$, and $p(x)\leq q(x)$
for all $x\in \R^n$.  Given constants
$a,\,b$ that satisfy
\[
-\frac{n}{q_{+}} < a\leq  b<\frac{n}{(p_{-})^{\prime}},
\]
suppose further that  $\frac{1}{p(x)}-\frac{1}{q(x)}=\frac{1}{n}+\frac{a-b}{n}$.
Then for all  $f\in C_0^\infty(\R^n)$,
\begin{equation}\label{hardysobolev}
  \left\||\cdot|^{a}f\right\|_{q(\cdot)}
  \leq C\||\cdot|^{b}\nabla f\|_{p(\cdot)},
\end{equation}
where $C$ is a positive constant independent of $f$.
\end{thm}

\begin{remark}
  If we set $\pp=\qq$, $b=0$ and $a=-1$ in
  Theorem~\ref{Hardy-Sob}, we immediately get Theorem~\ref{thm1}.  Note that these inequalities were studied in \cite{RS21} in generalized grand Lebesgue spaces. 
\end{remark}

\begin{remark}
  When $a=b=0$, Theorem~\ref{Hardy-Sob} becomes the version of the
  Hardy-Sobolev inequality in variable Lebesgue spaces proved
  in~\cite[Section~6.4]{CUF13}.  However, there this theorem included 
  the case $p_-=1$; the proof depended on weighted norm inequalities
  in the constant exponent case, the theory of Rubio de Francia extrapolation
  and a technique due to Maz'ya.  It is an open question whether
  Theorem~\ref{Hardy-Sob} can be extended to include the case $p_-=1$. 
\end{remark}

As a consequence of Theorem~\ref{Hardy-Sob}, or more precisely, by
adapting its proof, we prove a weighted
Gagliardo-Nirenberg inequality in the variable Lebesgue spaces.  In the constant exponent case, this result was studied in  \cite{RSY17,RSY18}.  For
related results without weights, see~\cite[Section~6.5.10]{CUF13}.

\begin{thm} \label{thm:wGN}
Fix an exponent $\pp \in \Pp(\R^n)$ such that 
$1<p_-\leq p_+<n$ and $\pp\in LH(\R^n)$.  Define $\ps \in \Pp(\R^n)$ by
\[ \frac{1}{p(x)}-\frac{1}{p^*(x)} = \frac{1}{n}. \]
Fix a constant $a$ such that
\[ -\frac{n}{(p^*)_+} < a < \frac{n}{(p_-)'}. \]
  Let $\qq \in \Pp(\R^n)$, fix $\theta \in [0,1]$, and define $\rr\in
  \Pp(\R^n)$ by
  \[ \frac{1}{r(x)} = \frac{\theta}{p^*(x)}+\frac{1-\theta}{q(x)}. \]
  Then for all $f\in C_0^\infty(\R^n)$, 
\begin{equation}\label{GNi}
  \left\||\cdot|^a f\right\|_{\rr}
  \leq
  C\||\cdot|^a\nabla f\|^{\theta}_{\pp}\;\||\cdot|^a f\|^{1-\theta}_{\qq},
\end{equation}
where $C$ is a positive constant independent of $f$.
\end{thm}

We can also prove Poincar\'e-type inequalities analogous to the inequalities in Theorem~\ref{hardysobolev}.  Given a set $\Omega$ such that $0<|\Omega|<\infty$ and a locally integrable function $f$, define
\[ \langle f\rangle_\Omega 
= \frac{1}{|\Omega|}\int_\Omega f(x)\,dx. \]

\begin{thm}[Poincar\'e inequality]\label{thm:poincare}
Let $\Omega$ be a bounded, convex set.  Fix exponents $\pp,\, \qq \in \Pp(\Omega)$ such that $\pp,\, \qq \in LH(\Omega)$,
$1<p_-\leq p_+<\infty$, $1<q_-\leq q_+<\infty$, and $p(x)\leq q(x)$
for all $x\in \R^n$.  Given constants
$a,\,b$ that satisfy
\[
-\frac{n}{q_{+}} < a\leq  b<\frac{n}{(p_{-})^{\prime}},
\]
suppose further that  $\frac{1}{p(x)}-\frac{1}{q(x)} \leq \frac{1}{n}+\frac{a-b}{n}$.
Then for all  $f\in C^\infty(\Omega)$,
\begin{equation}\label{eqn:poincare1}
  \left\||\cdot|^{a}[f-\langle f\rangle_\Omega ]\right\|_{L^{q(\cdot)}(\Omega)}
  \leq C\||\cdot|^{b}\nabla f\|_{L^{p(\cdot)}(\Omega)}
\end{equation}
where $C$ is a positive constant independent of $f$.
\end{thm}

\begin{remark} \label{remark:LH-omega}
The assumption that $\pp,\,\qq \in LH(\Omega)$ can be replaced with the seemingly stronger assumption that $\pp,\,\qq \in LH(\R^n)$, since log-H\"older continuous functions on any set $\Omega$ can be extended to functions in $LH(\R^n)$:  See~\cite[Lemma~2.4]{CUF13}.
\end{remark}

\begin{remark} When $\pp=\qq$ and $a=b=0$, this result was proved in~\cite[Theorem~6.21]{CUF13}, but again, this result included the case $p_-=1$.
\end{remark}

\medskip

Our next results are the analogous theorems for the fractional Laplacian.

\begin{thm}[Fractional Hardy-Rellich
  inequality]\label{frac-Hardy-Leray}
  Fix $s \in [0,1]$.  Given $\pp\in \Pp(\R^n)$, suppose $1<p_-\leq
  p_+<\frac{n}{2s}$ and $\pp \in LH(\R^n)$.   Then for all $f\in C_0^\infty(\R^n)$, 
\begin{equation}\label{hardyin}
\left\||\cdot|^{-2s}f\right\|_{\pp}\leq C\|(-\Delta )^{s}f\|_{\pp},
\end{equation}
where $C$ is a positive constant independent of $f$. 
\end{thm}

Theorem~\ref{frac-Hardy-Leray} is a special case of the following result.

\begin{thm}[Fractional Hardy-Sobolev inequality]\label{frac-Hardy-Sob}
 Fix $s \in [0,1]$.  Given $\pp, \, \qq\in \Pp(\R^n)$, suppose that
 $\pp,\,\qq \in LH(\R^n)$, $1<p_-\leq p_+<\infty$, $1<q_-\leq
 q_+<\infty$, and $p(x)\leq q(x)$ for all $x\in \R^n$.  Given
 constants $a,\,b$ that satisfy
 \[ - \frac{n}{q_+} <a \leq b < \frac{n}{(p_-)'}, \]
 suppose further that $\frac{1}{p(x)}-\frac{1}{q(x)}=\frac{2s}{n}+
 \frac{a-b}{n}$.  Then for all $f\in C_0^\infty(\R^n)$, 
\begin{equation}\label{frachardysobin}
\left\||\cdot|^{a}f\right\|_{\qq}\leq C\||\cdot|^{b}(-\Delta )^{s} f\|_{\pp},
\end{equation}
where $C$ is a positive constant independent of $f$.
\end{thm}

Finally, by adapting the proof of Theorem~\ref{Hardy-Sob}, we prove a weighted
fractional Gagliardo-Nirenberg inequality in the variable Lebesgue spaces. Even in the constant exponent case, this result appears to be new.

\begin{thm}  \label{thm:frac-GNi}
Fix $s\in [0,1]$ and fix an exponent $\pp \in \Pp(\R^n)$ such that
$1<p_-\leq p_+ < \frac{n}{2s}$ and $\pp \in LH(\R^n)$.  Define
$p_s(\cdot)\in \Pp(\R^n)$ by
\[ \frac{1}{p(x)}- \frac{1}{p_s(x)} = \frac{2s}{n}. \]
Fix a constant $a$ such that
\[ -\frac{n}{(p_s)_+} < a < \frac{n}{(p_-)'}.  \]
Let $\qq \in \Pp(\R^n)$, fix $\theta \in [0,1]$, and define $\rr\in
\Pp(\R^n)$ by
\[ \frac{1}{r(x)} = \frac{\theta}{p_s(x)} + \frac{1-\theta}{q(x)}. \]
Then for all $f\in C_0^\infty(\R^n)$, 
\begin{equation}\label{eqn:frac-GNi1}
  \left\||\cdot|^af\right\|_{\rr}
  \leq
  C\|(-\Delta )^{s} f\|^{\theta}_{\pp}\|f\|^{1-\theta}_{\qq},
\end{equation}
where $C$ is a positive constant independent of $f$.
\end{thm}

\medskip

The remainder of this paper is organized as follows.  In
Section~\ref{section:prelim} we gather some basic definitions and
results about variable Lebesgue spaces that we will need.  The heart of the proofs
of our main results is a generalization of the Stein-Weiss theorem for
the fractional integral operator to the scale of variable Lebesgue
spaces, which is of independent interest.  We state and prove this result in
Section~\ref{section:stein-weiss}.  In
Section~\ref{section:proofs} we give the proofs of
Theorems~\ref{thm1}--\ref{thm:frac-GNi}.  Finally, in Section~\ref{section:pde-app} we give an application of our results: we show that a Neumann-type problem for the degenerate  operator 
\[ Lu = - \Div(|\sqrt{Q} \grad u|^{\pp-2}Q\grad u) \]
has a solution. This operator has been studied previously in~\cite{MR2670139,MR3585054,MR3974098,MR4332462}.  It is a generalization of the $\pp$-Laplacian
\[ \Delta_\pp u = 
-\Div(|\grad u|^{\pp-2}\grad u), \]
which arises in the calculus of variations as an example of
nonstandard growth conditions, and has been studied by a
number of authors: see~\cite{MR3308513, MR2639204,
  MR2291779, MR3379920} and the references they contain.

Throughout this paper, $n$
will denote the dimension of the underlying space $\R^n$.   By a
constant $C$ we will mean a value that may depend on the underlying
parameters (such as the exponent functions) but not on a particular
function $f$.  Its value may change from line to line.  Sometimes, we
will write $A\lesssim B$ instead of $A\leq cB$.  If $A\lesssim B$ and
$B\lesssim A$, we will write $A\approx B$. 

\section{Properties of variable Lebesgue spaces}
\label{section:prelim}

In this section we give some basic definitions and properties of
variable Lebesgue spaces.  For proofs and further information,
see~\cite[Chapter~2]{CUF13}.

Given a set $\Omega\subset \R^n$, let $\mathcal{P}(\Omega)$ be the set
of all Lebesgue measurable functions, denoted by $p(\cdot)$, such that
$p(\cdot):\; \Omega \rightarrow[1, \infty]$. The functions $p(\cdot)$
are referred to as  exponent functions.  Given
$p(\cdot) \in \mathcal{P}(\Omega)$ and a set $E \subset \Omega$, define
\[ p_{-}(E)=\essinf_{x\in E} p(x), 
  \quad
  p_{+}(E)=\esssup_{x\in E} p(x). \]
If the domain is clear we will simply write
$p_{-}=p_{-}(\Omega), p_{+}=p_{+}(\Omega).$

Given $p(\cdot)\in \Pp(\R^n)$, define the conjugate exponent
$p^{\prime}(\cdot)$ pointwise by
$\frac{1}{p^{\prime}(x)}+\frac{1}{p(x)}=1$, with the conventions that
$1/\infty=0$ and $1/0=\infty$.   The following identities
follow at once from the definition:
\[ \left(p^{\prime}(\cdot)\right)_{+}=\left(p_{-}\right)^{\prime},
  \quad
  \left(p^{\prime}(\cdot)\right)_{-}=\left(p_{+}\right)^{\prime}. \]
	
Given $\pp \in \Pp(\Omega)$, we say that it belongs to the class of
locally log-H\"older continuous exponents $L H_0(\Omega) $, denoted by
$p(\cdot) \in L H_0(\Omega) $, if there exists a constant $C_{0}$ such that
\[ |p(x)-p(y)| \leq \frac{C_0}{-\log (|x-y|)} \]
for all $x,y \in \Omega$ with $ |x-y|<\frac{1}{2}$.  We say that $\pp$
belongs to the class of exponents that are log-H\"older continuous at
infinity, $L H_{\infty}(\Omega)$, denoted by $p(\cdot) \in L H_{\infty}(\Omega)$, if
there exists a constant $C_{\infty}$ and $p_{\infty}$ such that
$$
\left|p(x)-p_{\infty}\right| \leq \frac{C_{\infty}}{\log (e+|x|)}
$$
for all $x \in \Omega$.  Define $LH(\Omega)=LH_0(\Omega)\cap
LH_\infty(\Omega)$. 

Given $p(\cdot) \in \mathcal{P}(\Omega)$ and a Lebesgue
  measurable function $f$, define the modular functional (or simply
  the modular) associated with $p(\cdot)$ by
$$
\rho_{p(\cdot), \Omega}(f)
=\int_{\Omega\backslash \Omega_{\infty}}|f(x)|^{p(x)} d x
+\|f\|_{L^{\infty}\left(\Omega_{\infty}\right)},
$$
where $\Omega_\infty= \{ x\in \Omega : p(x)=\infty\}$.  If
$\left|\Omega_{\infty}\right|=0$, in particular when $p_{+}<\infty$,
$\|f\|_{L^{\infty}\left(\Omega_{\infty}\right)}=0$; when
$\left|\Omega \backslash \Omega_{\infty}\right|=0$, then
$\rho_{p(\cdot),
  \Omega}(f)=\|f\|_{L^{\infty}\left(\Omega_{\infty}\right)}$. In
situations where there is no ambiguity we will simply write
$\rho_{p(\cdot)}(f)$ or $\rho(f)$.  

Define
  $L^{p(\cdot)}(\Omega)$ to be the set of Lebesgue measurable
  functions $f$ such that $\rho(f / \lambda)<\infty$ for some
  $\lambda>0$. Define $L_{\text {loc }}^{p(\cdot)}(\Omega)$ to be the
  set of measurable functions $f$ such that $f \in L^{p(\cdot)}(K)$
  for every compact set $K \subset \Omega$.
Define 
$$
\|f\|_{L^{p(\cdot)}(\Omega)}
=
\inf \left\{\lambda>0: \rho_{p(\cdot), \Omega}(f / \lambda) \leq 1\right\} .
$$
If the set on the righthand side is empty define
$\|f\|_{L^{p(\cdot)}(\Omega)}=\infty$. If there is no ambiguity over
the domain $\Omega$, we will often write $\|f\|_{p(\cdot)}$ instead of
$\|f\|_{L^{p(\cdot)}(\Omega)}$.  When $\pp$ equals a constant $p$,
then $\Lp(\Omega)$ is equal to the classical Lebesgue space
$L^p(\Omega)$ and $\|f\|_{\Lp(\Omega)}=\|f\|_{L^p(\Omega)}$. 

The functional $\|\cdot\|_{\Lp(\Omega)}$ is a norm, and equipped with
this norm, $\Lp(\Omega)$ is a Banach function space.
Given $\tau\geq
1$, the norm satisfies the rescaling property that for $\tau>1$,
\begin{equation} \label{eqn:rescale}
 \|f\|_{L^{\tau\pp}(\Omega)} = \||f|^\tau\|_{\Lp(\Omega)}^\tau. 
\end{equation} 
It also satisfies a generalized H\"older inequality:  given
$\pp,\,\qq,\,\rr \in \Pp(\Omega)$ such that for all $x\in \Omega$,
\[  \frac{1}{r(x)} = \frac{1}{p(x)}+\frac{1}{q(x)}, \]
then there exists a constant $C$ such that for all $f\in \Lp(\Omega)$ and $g\in \Lq(\Omega)$,
\begin{equation} \label{eqn:holder}
 \|fg\|_{L^\rr(\Omega)}
  \leq
  C\|f\|_{\Lp(\Omega)} \|g\|_{\Lq(\Omega)}. 
\end{equation}

If $\pp \in LH(\Omega)$, then for every cube $Q\subset \Omega$ we can
estimate the norm of the characteristic function of $Q$ in
$\Lp(\Omega)$ by
\begin{equation} \label{eqn:h-mean}
    \|\chi_Q\|_{\Lp(\Omega)} \approx |Q|^{\frac{1}{p_Q}},
\end{equation}
where $p_Q$ is the harmonic mean of $\pp$ on $Q$:
\[  \frac{1}{p_Q} = \avgint_Q \frac{dx}{p(x)}. \]
See~\cite[Section~4.6.2]{CUF13}.

\section{The Stein-Weiss inequality in variable Lebesgue spaces}
\label{section:stein-weiss}

At the heart of the proofs of our main results is an application of
the theory of weighted norm inequalities in the variable Lebesgue
spaces.  We make use of the following relationships between the
gradient, the fractional Laplacian and the Riesz potentials.  Given
$0<\alpha<n$, define the Riesz potential $I_\alpha$ to be the integral operator
\[ I_{\alpha}f(x)
  =
  \int_{\mathbb{R}^{n}}\frac{1}{|x-y|^{n-\alpha}}f(y)dy. \]
It is well-known (see, for instance,~\cite{MR0290095}) that
$I_\alpha : L^p(\R^n)\rightarrow L^q(\R^n)$, whenever
$1<p<\frac{n}{\alpha}$ and
$\frac{1}{p}-\frac{1}{q}=\frac{\alpha}{n}$.  The analogous inequality
holds in the variable Lebesgue spaces:  see~\cite[Chapter~5]{CUF13}. 

The connection between the Riesz potential and the gradient is given
by two inequalities.  The first is
\begin{equation} \label{eqn:rp1}
  |f(x)| \lesssim I_1(|\grad f|)(x), 
\end{equation}
which holds for all $f\in C_0^\infty(\R^n)$.  The second holds given any bounded, convex domain $\Omega$ and $x\in \Omega$:
\begin{equation} \label{eqn:rp1bis}
|f(x)-\langle f \rangle_\Omega|
\leq I_1(|\grad f|\chi_\Omega)(x). 
\end{equation}
(See~\cite{CUF13,MR1814364}.) 
Similarly, by the definition of the
fractional Laplacian and the properties of the Fourier transform, we
have that for $s \in [0,1]$,
\begin{equation} \label{eqn:rp2}
  I_{2s} ((-\Delta)^s f)(x) = f(x).
\end{equation}

We exploit these estimates using the following result, which is a
generalization of the classical Stein-Weiss inequality~\cite{StWe58}
to the variable Lebesgue spaces.

\begin{thm} \label{thm:var-stein-weiss}
Fix $0<\alpha<n$.  Given $\pp,\,\qq \in \Pp(\R^n)$, suppose that
$\pp,\,\qq \in LH(\R^n)$, $1<p_-\leq p_+<\infty$, $1<q_-\leq
q_+<\infty$, and $p(x)\leq q(x)$ for all $x\in \R^n$.  Given constants
$a,\, b$ that satisfy
\begin{equation} \label{eqn:vsw1}
 -\frac{n}{q_+} < a\leq b < \frac{n}{(p_-)'}, 
\end{equation}
suppose further that for all $x\in \R^n$, 
\begin{equation} \label{eqn:vsw2}
  \frac{1}{p(x)}-\frac{1}{q(x)} = \frac{\alpha}{n} +
  \frac{a-b}{n}.
\end{equation}
Then for all $f$ such that $|\cdot|^bf\in \Lp(\R^n)$,
\[ \||\cdot|^a I_\alpha f\|_{\qq} \leq C \||\cdot|^b f\|_{\pp}. \]
\end{thm}

We will prove Theorem~\ref{thm:var-stein-weiss} as a consequence of a
more general, two weight norm inequality in variable Lebesgue spaces
due to Melchiori and Pradolini~\cite[Theorem 1.1]{MP18}.  By a weight
$w$ we mean a non-negative, locally integrable function that satisfies
$0<w(x)<\infty$ almost everywhere.
		
\begin{thm}\label{thm:mp}
  Fix $0<\alpha<n$.  Given $\pp,\,\qq \in \Pp(\R^n)$, suppose that
$\pp,\,\qq \in LH(\R^n)$, $1<p_-\leq p_+<\infty$, $1<q_-\leq
q_+<\infty$, and $p(x)\leq q(x)$ for all $x\in \R^n$.
If for some $r,\,s>1$, the pair of weights $(u, v)$ satisfies

\begin{equation} \label{eqn:mp1}
\sup_Q \;|Q|^{\frac{\alpha}{n}}
\frac{\left\|\chi_Q\right\|_{q(\cdot)}}{\left\|\chi_Q\right\|_{p(\cdot)}}
\frac{\left\|\chi_Q u\right\|_{r q^{+}}}{\left\|
\chi_Q\right\|_{r q^{+}}}
\frac{\left\|\chi_Q  v^{-1}\right\|_{s\left(p^{-}\right)^{\prime}}}
{\left\|\chi_Q\right\|_{s\left(p^{-}\right)^{\prime}}}
<\infty.
\end{equation}
Then for all $f$ such that $vf\in \Lp(\R^n)$,
\[ \left\|u I_\alpha f\right\|_{q(\cdot)} \lesssim \|v
  f\|_{p(\cdot)}. \]
\end{thm}

To prove Theorem~\ref{thm:var-stein-weiss} we need one lemma.

\begin{lem} \label{lemma1} Fix $-n<\gamma<\infty$. Given any cube
  $Q=Q(l(Q), c_{Q})$, if $\dist(Q,0)\leq \ell(Q)$, then
  $$\bigg(\dashint_{Q}|x|^{\gamma}\,dx \bigg)^{\frac{1}{\gamma}}
  \lesssim \ell (Q).$$
If $\dist(Q,0)\geq \ell(Q)$, then
$$\bigg(\dashint_{Q}|x|^{\gamma}\,dx\bigg)^{\frac{1}{\gamma}}
\lesssim |c_{Q}|.$$
In both cases, the implicit constants depend only on $n$ and $\gamma$.
\end{lem}

\begin{proof}
  First assume that $\dist(Q,0)\leq \ell(Q)$. In this case there
  exists a ball $B_{Q}=B(0,r)$, where $r\leq (\sqrt{n}+1)\,l(Q)$ and
  $Q<B_{Q}.$
Therefore, if we shift to polar coordinates,
\begin{align*}
  \dashint_{Q}|x|^{\gamma}\,dx
  & \leq c(n) \dashint_{B_{Q}}|x|^{\gamma}\,dx  
\\ &  \approx \ell(Q)^{-n} \int_{0}^{(\sqrt{n}+1)\,\ell(Q)} \int_{S^{n-1}}r^{\gamma} r^{n-1}\,d\theta dr  
\\ &  \approx \ell(Q)^{-n} \int_{0}^{(\sqrt{n}+1)\,\ell(Q)} r^{\gamma+n-1} \,dr  
  \\ &  \approx \ell(Q)^{-n} \ell(Q)^{\gamma+n}
       \\& =\ell(Q)^{\gamma}.  
\end{align*}

Now suppose that $\text{dist} (Q,0)\geq \ell(Q)$. Let $x_{0}\in Q$ be
such that $|x_{0}|=\text{dist} (Q,0).$ Given any point $x\in Q$, we
have
\[ |x_{0}| \leq |x|  \leq |x_{0}|+|x-x_{0}|  
\leq  |x_{0}|+\sqrt{n}\,\ell(Q)\leq (\sqrt{n}+1) |x_{0}|.   \]
A similar computation shows $|x_0|\approx |c_{Q}|$. Therefore, we have that
$$\dashint_{Q} |x|^{\gamma}\,dx \lesssim |c_{Q}|^{\gamma}.$$
\end{proof}

\begin{proof}[Proof of Theorem~\ref{thm:var-stein-weiss}]
We need to show that if we let  $u(x)=|x|^{a}$ and $v(x)=|x|^{b}$,
then with our hypotheses on $a,\,b$ and $\pp,\,\qq$, the
condition~\eqref{eqn:mp1} holds.  Our first step is to replace this
condition with a simpler one.  Since $\pp,\qq\in
LH(\R^n)$, if we apply~\eqref{eqn:h-mean} to the norm of each
characteristic function, and rewrite the norms of the weights in terms
of Lebesgue spaces norms, we get that it is equivalent to assuming
that for some $r,\,s>1$,
\begin{equation} \label{eqn:mp2}
  \sup_Q |Q|^{\frac{\alpha}{n}+\frac{1}{q_{Q}}-\frac{1}{p_{Q}}}
  \left(\dashint_{Q} |x|^{arq_{+}} \,dx \right)^{\frac{1}{rq_{+}}}
\left(\dashint_{Q} |x|^{-bs(p_{-})^{\prime}} \,dx
\right)^{\frac{1}{s(p_{-})^{\prime}}}
< \infty. 
\end{equation}
(We note in passing that this condition should be compared to the
result of Perez \cite{Perez94} in the constant exponent case.)

Fix a cube $Q$.   We first choose $r,\,s>1$ such that
\[ \int_{Q} |x|^{arq_{+}} \,dx <  \infty,
  \quad
\int_{Q} |x|^{-bs(p_{-})^{\prime}} \,dx < \infty.  \]
This is possible by assumption~\eqref{eqn:vsw1}.  

Now suppose first that
$\dist (Q,0)\leq \ell(Q)$.
Then by Lemma \ref{lemma1}, the lefthand side of  \eqref{eqn:mp2} is
bounded above by.
\begin{equation}\label{cond3}
|Q|^{\frac{\alpha}{n}+\frac{1}{q_{Q}}-\frac{1}{p_{Q}}}|Q|^{\frac{a}{n}}|Q|^{-\frac{b}{n}}=
|Q|^{\frac{\alpha}{n}+\frac{1}{q_{Q}}-\frac{1}{p_{Q}}+\frac{a-b}{n}}
\end{equation}
However, if we take the integral average of \eqref{eqn:vsw2} over the cube $Q$, we get
that
\[ \frac{\alpha}{n}+\frac{1}{q_{Q}}-\frac{1}{p_{Q}}+\frac{a-b}{n}  =
  0. \]
Therefore, for all such cubes, the  lefthand side of  \eqref{eqn:mp2}
is uniformly bounded.

Now suppose that  $\dist (Q,0)\geq \ell(Q)$.  Then the lefthand
side of~\eqref{eqn:mp2} is bounded above by 
\[
|c_{Q}|^{\alpha+\frac{n}{q_{Q}}-\frac{n}{p_{Q}}}|c_{Q}|^{a}|c_{Q}|^{-b}=
|c_{Q}|^{\alpha+\frac{n}{q_{Q}}-\frac{n}{p_{Q}}+a-b}.
\]
But arguing as above we must have that the exponent on the righthand
side is equal to $0$, and so again lefthand side of  \eqref{eqn:mp2}
is uniformly bounded for all such cubes.  Therefore, \eqref{eqn:mp2}
holds, and so by Theorem~\ref{thm:mp} we get the desired inequality
for $I_\alpha$.
\end{proof}

\medskip

\begin{remark}
 In the proof of Theorem~\ref{thm:var-stein-weiss} we do not actually
 use the hypothesis $a\leq b$.  However, if the opposite inequality
 holds, then~\eqref{eqn:vsw2} implies that
 \[ \frac{1}{p(x)}-\frac{1}{q(x)} > \frac{\alpha}{n}. \]
 In the constant exponent case, it is known that in this case the
 weighted norm inequality only holds in the trivial case when $u(x)=0$
 whenever $v(x)<\infty$ (see the remarks after \cite[Theorem
 1]{Saw82}).
\end{remark}

\section{Proofs of the main results}
\label{section:proofs}

We can now prove our main results.  As we noted in the Introduction,
Theorem~\ref{thm1} is a special case of Theorem~\ref{Hardy-Sob}.

\begin{proof}[Proof of Theorem~\ref{Hardy-Sob}]
  Fix $f\in C_c^\infty(\R^n)$.  Then
  \[ \| |\cdot|^a f\|_\qq
    \leq
    C  \| |\cdot|^a I_1(|\grad f|)\|_\qq
    \leq
    C \| |\cdot|^b \grad f\|_\pp; \]
  the first inequality follows from~\eqref{eqn:rp1} and the second
  from our hypotheses and Theorem~\ref{thm:var-stein-weiss} with
  $\alpha=1$. 
\end{proof}

\begin{proof}[Proof of Theorem~\ref{thm:wGN}]
  Fix $f\in C_c^\infty(\R^n)$.  Then by H\"older's
  inequality~\eqref{eqn:holder} and the rescaling
  property~\eqref{eqn:rescale},
  \[ \||\cdot|^a f\|_\rr
    \leq C\||\cdot|^a  f\|_\ps^{\theta}
    \||\cdot|^a f\|_\qq^{1-\theta}.  \]
  By inequality~\eqref{eqn:rp1} and by
  Theorem~\ref{thm:var-stein-weiss} with $\alpha=1$, $a=b$ and $\qq=\ps$, which we can apply because of our
  hypotheses, 
  \[ \||\cdot|^a  f\|_\ps
    \leq
    C\||\cdot|^a  I_1(|\grad f)\|_\ps
    \leq
    C\| |\cdot|^a \grad f \|_\pp. \]
  If we combine these two estimates we get the desired result.
\end{proof}

\medskip

\begin{proof}[Proof of Theorem~\ref{thm:poincare}]
The proof is nearly the same as the proof of Theorem~\ref{Hardy-Sob}.  As we noted  in Remark~\ref{remark:LH-omega}, we may assume without loss of generality that $\pp,\, \qq \in LH(\R^n)$.  Define $\rr\in LH(\R^n)$ by
\[ \frac{1}{p(x)} - \frac{1}{r(x)}
= \frac{1}{n} + \frac{a-b}{n}. \]
Then $q(x)\leq r(x)$, so we can define the exponent $\sst$ by 
\[ \frac{1}{q(x)}=\frac{1}{s(x)}+\frac{1}{r(x)}.  \]
By the generalized H\"older inequality,
\begin{align*}
\left\||\cdot|^{a}[f-\langle f\rangle_\Omega ]\right\|_{L^{q(\cdot)}(\Omega)} 
& \leq 
\left\||\cdot|^{a}[f-\langle f\rangle_\Omega ]\right\|_{L^{r(\cdot)}(\Omega)}
\|\chi_\Omega\|_{L^{s(\cdot)}(\Omega)} \\
& = C(\Omega, \sst) 
\left\||\cdot|^{a}[f-\langle f\rangle_\Omega ]\right\|_{L^{r(\cdot)}(\Omega)}. 
\end{align*}
By inequality~\eqref{eqn:rp1bis} and Theorem~\ref{thm:var-stein-weiss}, 
\begin{equation*}
\left\||\cdot|^{a}[f-\langle f\rangle_\Omega ]\right\|_{L^{r(\cdot)}(\Omega)}
 \leq \left\||\cdot|^{a}I_1(|\grad f|\chi_\Omega) ]\right\|_{L^{r(\cdot)}(\Omega)} 
\leq C\| |\cdot|^b \grad f \|_{L^{p(\cdot)}(\Omega)}.
\end{equation*} 
  If we combine these inequalities we get the desired result.
\end{proof}
\medskip

As we noted in the Introduction, Theorem~\ref{frac-Hardy-Leray} is a
special case of Theorem~\ref{frac-Hardy-Sob} with $a=2s$, $b=0$, and
$\pp=\qq$.

\begin{proof}[Proof of Theorem~\ref{frac-Hardy-Sob}]
  The proof is essentially the same as the proof of
  Theorem~\ref{Hardy-Sob}, but we use~\eqref{eqn:rp2} instead of
  \eqref{eqn:rp1}, and apply Theorem~\ref{thm:var-stein-weiss} with
  $\alpha=2s$.
\end{proof}

\begin{proof}[Proof of Theorem~
\ref{thm:frac-GNi}]
The proof is essentially the same as the proof of
Corollary~\ref{thm:wGN} but again using~\eqref{eqn:rp2} and
Theorem~\ref{thm:var-stein-weiss} with $\alpha=2s$, $a=b$ and
$\qq=p_s(\cdot)$. 
\end{proof}

\section{Solutions to a degenerate $\pp$-Laplacian}
\label{section:pde-app}

In this section we give an application of Theorem~\ref{thm:poincare} to show the existence of solutions to a degenerate Neumann-type problem studied in~\cite{MR4332462}.  Let $\Omega$ be a bounded, open domain in $\R^n$, and let $Q$ be an $n\times n$ self-adjoint, positive-semi-definite matrix function defined on an open neighborhood of $\bar{\Omega}$, and let $v$ be a non-negative, locally integrable function also defined on an open neighborhood of $\bar{\Omega}$.  We are interested in the existence of weak solutions to the problem

\begin{equation}\label{nprob}
\begin{cases}
\Div\Big(\Big|\sqrt{Q}\nabla u\Big|^{\pp-2} Q\nabla u\Big) 
& = |f|^{\pp-2}fv^{\pp} \text{ \; in }\Omega\\
{\bf n}^T \cdot Q \nabla u &= 0\text{ \; on }\partial \Omega,
\end{cases}
\end{equation}
where ${\bf n}$ is the outward unit normal vector of
$\partial \Omega$.  The precise definition of a weak solution is somewhat technical, and we refer the reader to~\cite{MR4332462} for complete details.  (We note in passing that this definition extends to the case when $\Omega$ has a rough boundary and the unit normal is not well-defined.)

In~\cite{MR4332462} the authors showed that if $1<p_-\leq p_+ < \infty$, $v,\|\sqrt{Q(\cdot)}|_{op} \in \Lp(\Omega)$, where for each $x\in \Omega$ $|\sqrt{Q(x)}|_{op}$ is the operator norm of the matrix $\sqrt{Q(x)}$, then the existence of a solution to~\eqref{nprob} was equivalent to the existence of a degenerate Poincar\'e inequality of the form
\begin{equation} \label{eqn:degen-poincare}
\| v[ f-\langle f \rangle_{\Omega, v}] \|_{\Lp(\Omega)}
\leq
C \| \sqrt{Q}\grad f\|_{\Lp(\Omega)},
\end{equation}
where 
\[ \langle f \rangle_{\Omega, v}
= \frac{1}{v(\Omega)}\int_\Omega f(x)v(x)\,dx. \]
By using this result and Theorem~\ref{thm:poincare} we can show that solutions exist when $v$ is a power weight and the smallest eigenvalue of $Q$ is bounded below by a power weight.

\begin{thm} \label{thm:var-neumann}
Fix an exponent $\pp \in \Pp(\Omega)$  such that $1<p_-\leq p_+<\infty$ and  $\pp \in LH(\Omega)$.  Let $a,\,b \in \R$ be such that 
\[ -\frac{n}{p_+} < a\leq b < \frac{n}{(p_-)'} \quad \text{and} \quad
\frac{1}{n}+\frac{a-b}{n}\geq 0. \]
Let $Q$ be a self-adjoint, positive semi-definite matrix function such that $|\sqrt{Q}|_{\op} \in \Lp(\Omega)$, and for all $x\in \Omega$ and $\xi \in \R^n$,
\[ c|x|^b|\xi|\leq |\sqrt{Q(x)}\xi|. \]
Let $v(x)=|x|^a$.  Then there exists a solution to the Neumann-type problem~\eqref{nprob}.
\end{thm}

\begin{proof}
By~\cite[Theorem~1.8]{MR4332462} the existence of solutions follows immediately if we have that $v\in \Lp(\Omega)$ and the Poincar\'e inequality~\eqref{eqn:degen-poincare} holds.  The first is immediate:  since $\Omega$ is bounded, by~\cite[Corollary~2.48]{CUF13},
\[ \||\cdot|^a\|_{\Lp(\Omega)}
\leq (1+|\Omega|)
\| |\cdot|^a\|_{L^{p_+}(\Omega)}<\infty; \]
the final inequality follows from our assumptions on $a$.  To prove the Poincar\'e inequality, note first that by H\"older's inequality in the variable Lebesgue spaces, 
\[ | \langle f \rangle_{\Omega} 
- \langle f \rangle_{\Omega,v}|
=
\bigg| \frac{1}{v(\Omega)}
\int_\Omega 
[\langle f \rangle_{\Omega} -f(x)\ v(x)\,dx \bigg|
\leq \frac{\|\chi_\Omega|_{L^\cpp(\Omega)}}{v(\Omega)}
\|f- \langle f \rangle_{\Omega} \|_{\Lp(\Omega)}. 
\]
Therefore, by Theorem~\ref{thm:poincare} and our assumptions on $Q$,
\begin{align*}
\| |\cdot|^a[ f-\langle f \rangle_{\Omega, v}] \|_{\Lp(\Omega)}
& \leq C
\| |\cdot|^a[ f-\langle f \rangle_{\Omega}] \|_{\Lp(\Omega)} \\
& \leq C\||\cdot|^b \grad f\|_{\Lp(\Omega)} \\
& \leq C \| \sqrt{Q}\grad f\|_{\Lp(\Omega)}.
\end{align*}
\end{proof}

\medskip

It is straightforward to construct nontrivial examples of exponents $\pp$ and matrices $Q$ for which solutions to~\eqref{nprob} exist.  For instance, in $\R^3$ let $\Omega=B(0,1)$, $a=-3/4$, and $b=1/4$.  Let $\pp\in \Pp(B(0,1))$ be any log-H\"older continuous function such that $p_-=2$ and $p_+=3$.  If we define
\[ Q(x) = \begin{pmatrix}
|x|^{-\frac{3}{4}} & 0 & 0 \\
0 & |x|^{-\frac{1}{4}} & 0 \\
0 & 0 & |x|^{\frac{1}{4}}
\end{pmatrix}, \]
then it is immediate that the hypotheses of Theorem~\ref{thm:var-neumann} are satisfied.

\bibliographystyle{plain}
\bibliography{hardy}

\begin{thebibliography}{10}

\bibitem{MR3626031}
B.~Abdellaoui and R.~Bentifour.
\newblock Caffarelli-{K}ohn-{N}irenberg type inequalities of fractional order
  with applications.
\newblock {\em J. Funct. Anal.}, 272(10):3998--4029, 2017.

\bibitem{MR768824}
L.~Caffarelli, R.~Kohn, and L.~Nirenberg.
\newblock First order interpolation inequalities with weights.
\newblock {\em Compositio Math.}, 53(3):259--275, 1984.

\bibitem{CUF13}
D.~Cruz-Uribe and A.~Fiorenza.
\newblock {\em Variable {L}ebesgue spaces}.
\newblock Applied and Numerical Harmonic Analysis. Birkh\"{a}user/Springer,
  Heidelberg, 2013.
\newblock Foundations and harmonic analysis.

\bibitem{MR4332462}
D.~Cruz-Uribe, M.~Penrod, and S.~Rodney.
\newblock Poincar\'{e} inequalities and {N}eumann problems for the variable
  exponent setting.
\newblock {\em Math. Eng.}, 4(5):Paper No. 036, 22, 2022.

\bibitem{deNitti-Djitte}
N.~de~Nitti and S.~M. Djitte.
\newblock Fractional hardy-rellich inequalities via a pohozaev identity.
\newblock {\em preprint}, 2022.

\bibitem{MR2944369}
E.~Di~Nezza, G.~Palatucci, and E.~Valdinoci.
\newblock Hitchhiker's guide to the fractional {S}obolev spaces.
\newblock {\em Bull. Sci. Math.}, 136(5):521--573, 2012.

\bibitem{diening-harjulehto-hasto-ruzicka2010}
L.~Diening, P.~Harjulehto, P.~H\"ast\"o, and M.~R{\r{u}}{\v{z}}i{\v{c}}ka.
\newblock {\em Lebesgue and {S}obolev spaces with {V}ariable {E}xponents},
  volume 2017 of {\em Lecture Notes in Mathematics}.
\newblock Springer, Heidelberg, 2011.

\bibitem{DieSam07-ext1D}
L.~Diening and S.~Samko.
\newblock Hardy inequality in variable exponent {L}ebesgue spaces.
\newblock {\em Fract. Calc. Appl. Anal.}, 10(1):1--18, 2007.

\bibitem{MR3308513}
D.~E. Edmunds, J.~Lang, and O.~M\'{e}ndez.
\newblock {\em Differential operators on spaces of variable integrability}.
\newblock World Scientific Publishing Co. Pte. Ltd., Hackensack, NJ, 2014.

\bibitem{MR2469027}
R.~L. Frank and R.~Seiringer.
\newblock Non-linear ground state representations and sharp {H}ardy
  inequalities.
\newblock {\em J. Funct. Anal.}, 255(12):3407--3430, 2008.

\bibitem{MR1814364}
D.~Gilbarg and N.~S. Trudinger.
\newblock {\em Elliptic {P}artial {D}ifferential {E}quations of {S}econd
  {O}rder}.
\newblock Classics in Mathematics. Springer-Verlag, Berlin, 2001.
\newblock Reprint of the 1998 edition.

\bibitem{zbMATH02591430}
G.~H. Hardy.
\newblock Notes on some points in the integral calculus. {LX}.
\newblock {\em Messenger of Math.}, 54:150--156, 1925.

\bibitem{HHK05}
P.~Harjulehto, P.~H\"{a}st\"{o}, and M.~Koskenoja.
\newblock Hardy's inequality in a variable exponent {S}obolev space.
\newblock {\em Georgian Math. J.}, 12(3):431--442, 2005.

\bibitem{MR2639204}
P.~Harjulehto, P.~H\"{a}st\"{o}, \'{U}t~V. L\^{e}, and M.~Nuortio.
\newblock Overview of differential equations with non-standard growth.
\newblock {\em Nonlinear Anal.}, 72(12):4551--4574, 2010.

\bibitem{MR436854}
I.~W. Herbst.
\newblock Spectral theory of the operator {$(p^{2}+m^{2})^{1/2}-Ze^{2}/r$}.
\newblock {\em Comm. Math. Phys.}, 53(3):285--294, 1977.

\bibitem{MR3585054}
K.~Ho and I.~Sim.
\newblock Existence results for degenerate {$p(x)$}-{L}aplace equations with
  {L}eray-{L}ions type operators.
\newblock {\em Sci. China Math.}, 60(1):133--146, 2017.

\bibitem{MR2670139}
Y.-H. Kim, L.~Wang, and C.~Zhang.
\newblock Global bifurcation for a class of degenerate elliptic equations with
  variable exponents.
\newblock {\em J. Math. Anal. Appl.}, 371(2):624--637, 2010.

\bibitem{KokSam04-1D}
V.~Kokilashvili and S.~Samko.
\newblock Maximal and fractional operators in weighted {$L^{p(x)}$} spaces.
\newblock {\em Rev. Mat. Iberoamericana}, 20(2):493--515, 2004.

\bibitem{MR3974098}
L.~Kong.
\newblock A degenerate elliptic system with variable exponents.
\newblock {\em Sci. China Math.}, 62(7):1373--1390, 2019.

\bibitem{MR1555394}
J.~Leray.
\newblock Sur le mouvement d'un liquide visqueux emplissant l'espace.
\newblock {\em Acta Math.}, 63(1):193--248, 1934.

\bibitem{MCMO07}
R.~A. Mashiyev, B.~\c{C}eki\c{c}, F.~I. Mamedov, and S.~Ogras.
\newblock Hardy's inequality in power-type weighted {$L^{p(\cdot)}(0,\infty)$}
  spaces.
\newblock {\em J. Math. Anal. Appl.}, 334(1):289--298, 2007.

\bibitem{MP18}
L.~Melchiori and G.~Pradolini.
\newblock Potential operators and their commutators acting between variable
  {L}ebesgue spaces with different weights.
\newblock {\em Integral Transforms Spec. Funct.}, 29(11):909--926, 2018.

\bibitem{MR2291779}
G.~Mingione.
\newblock Regularity of minima: an invitation to the dark side of the calculus
  of variations.
\newblock {\em Appl. Math.}, 51(4):355--426, 2006.

\bibitem{Perez94}
C.~P\'{e}rez.
\newblock Two weighted inequalities for potential and fractional type maximal
  operators.
\newblock {\em Indiana Univ. Math. J.}, 43(2):663--683, 1994.

\bibitem{RS21}
J.~E. Restrepo and D.~Suragan.
\newblock Hardy type inequalities in generalized grand {L}ebesgue spaces.
\newblock {\em Adv. Oper. Theory}, 6(2):Paper No. 30, 31, 2021.

\bibitem{MR3379920}
V.~R\u{a}dulescu and D.~Repov\v{s}.
\newblock {\em Partial differential equations with variable exponents}.
\newblock Monographs and Research Notes in Mathematics. CRC Press, Boca Raton,
  FL, 2015.
\newblock Variational methods and qualitative analysis.

\bibitem{RSY17}
M.~Ruzhansky, D.~Suragan, and N.~Yessirkegenov.
\newblock Extended {C}affarelli-{K}ohn-{N}irenberg inequalities and
  superweights for {L}p-weighted {H}ardy inequalities.
\newblock {\em C. R. Acad. Sci. Paris}, 355:694--698, 2017.

\bibitem{RSY18}
M.~Ruzhansky, D.~Suragan, and N.~Yessirkegenov.
\newblock Extended {C}affarelli-{K}ohn-{N}irenberg inequalities, and
  remainders, stability and superweights for {L}p-weighted {H}ardy
  inequalities.
\newblock {\em Trans. Amer. Math. Soc. Ser. B}, 5:32--62, 2018.

\bibitem{Saw82}
E.~T. Sawyer.
\newblock Two weight norm inequalities for certain maximal and integral
  operators.
\newblock In {\em Harmonic analysis ({M}inneapolis, {M}inn., 1981)}, volume 908
  of {\em Lecture Notes in Math.}, pages 102--127. Springer, Berlin-New York,
  1982.

\bibitem{MR0290095}
E.~M. Stein.
\newblock {\em Singular Integrals and Differentiability Properties of
  Functions}.
\newblock Princeton Mathematical Series, No. 30. Princeton University Press,
  Princeton, N.J., 1970.

\bibitem{StWe58}
E.~M. Stein and G.~Weiss.
\newblock Fractional integrals on {$n$}-dimensional {E}uclidean space.
\newblock {\em J. Math. Mech.}, 7:503--514, 1958.

\end{thebibliography}

\end{document}